\documentclass[11pt]{amsart}
\usepackage{amsmath}
\usepackage{a4wide}
\usepackage[utf8]{inputenc}
\usepackage{amssymb}
\usepackage{amsopn}
\usepackage{epsfig}
\usepackage{amsfonts}
\usepackage{latexsym}
\usepackage{amsthm}
\usepackage{enumerate}
\usepackage[UKenglish]{babel}
\usepackage{verbatim}
\usepackage{color}
\usepackage{pgf}
\usepackage{tikz}

\usepackage{mathrsfs}   % This provides an alternative way to typeset script capital letters, using \mathscr{}
 \usetikzlibrary{arrows,automata}

\newtheorem{theorem}{Theorem}[section]
\newtheorem{lemma}[theorem]{Lemma}
\newtheorem{proposition}[theorem]{Proposition}
\newtheorem{corollary}[theorem]{Corollary}

\theoremstyle{definition}
\newtheorem{definition}[theorem]{Definition}

\theoremstyle{remark}
\newtheorem{remark}[theorem]{Remark}

\numberwithin{equation}{section}

\newcommand{\R}{\ensuremath{\mathbb{R}}}
\newcommand{\N}{\ensuremath{\mathbb{N}}}
\renewcommand{\L}{\ensuremath{\mathcal{L}}}

\newcommand{\su}{\mathbf{u}}
\newcommand{\sv}{\mathbf{v}}
\newcommand{\sw}{ {\mathbf{w}}}
\newcommand{\sx}{ {\mathbf{x}}}
\newcommand{\sy}{ {\mathbf{y}}}

\renewcommand{\c}{ {\mathbf{c}}}
\renewcommand{\d}{ {\mathbf{d}}}

\newcommand{\bb}{\mathscr{B}}

\renewcommand{\u}{\ensuremath{\mathcal{U}}}
\newcommand{\ub}{\mathscr{U}}

\newcommand{\us}{\mathbf{U}}
\newcommand{\vs}{ {\mathbf{V}}}
\newcommand{\set}[1]{\left\{#1\right\}}
\newcommand{\la}{\lambda}

\newcommand{\f}{\infty}

\newcommand{\al}{\alpha}
\newcommand{\lle}{\preccurlyeq}
\newcommand{\lge}{\succcurlyeq}

\newcommand{\si}{\sigma}

\newcommand{\ra}{\rightarrow}
\begin{document}
\title{On the continuity of the Hausdorff dimension of the univoque set}

\author{Pieter Allaart}
\address[P. Allaart]{Mathematics Department, University of North Texas, 1155 Union Cir \#311430, Denton, TX 76203-5017, U.S.A.}
\email{allaart@unt.edu}

\author{Derong Kong}
\address[D. Kong]{Mathematical Institute, University of Leiden, PO Box 9512, 2300 RA Leiden, The Netherlands}
\email{d.kong@math.leidenuniv.nl}

\date{\today}
\dedicatory{}

%\begin{frontmatter}

\begin{abstract}
In a recent paper [{\em Adv. Math.}, 305:165--196, 2017], Komornik et al.~proved a long-conjectured formula for the Hausdorff dimension of the set $\mathcal{U}_q$ of numbers having a unique expansion in the (non-integer) base $q$, and showed that this Hausdorff dimension is continuous in $q$. Unfortunately, their proof contained a gap which appears difficult to fix. This article gives a completely different proof of these results, using a more direct combinatorial approach.
\end{abstract}
\maketitle

\section{Introduction}

Fix an integer $M\geq 1$ and a real number $q\in(1,M+1]$, and let $I_q:=[0,M/(q-1)]$. It is well known that every number $x\in I_q$ can be written in the form 
\begin{equation}
x=\sum_{j=1}^\infty \frac{x_j}{q^j}=:\pi_q(x_1 x_2\dots), \qquad x_j\in\{0,1,\dots,M\}\ \forall j.
\label{eq:q-expansion}
\end{equation}
We call such a representation a {\em $q$-expansion} of $x$. Such expansions were introduced by R\'enyi \cite{Renyi} and studied further by Parry \cite{Parry_1960}. They were then largely forgotten for about 30 years until Erd\H{o}s et al. \cite{EHJ,EJK} uncovered their fascinating mathematical structure. Since then, $q$-expansions have been the subject of a large number of research articles, many of which focus on the {\em univoque set}
\[
\u_q:=\{x\in I_q: x\ \mbox{has a unique $q$-expansion of the form \eqref{eq:q-expansion}}\}.
\]
This set was shown to be of Lebesgue measure zero in \cite{EJK}, but its more detailed structure was first exposed in the remarkable paper by Glendinning and Sidorov \cite{GlenSid}. For the case $M=1$, they found that phase transitions occur at two critical values $q_G:=(1+\sqrt{5})/2$ and $q_{KL}\approx1.78723$, as follows: $\u_q$ is (i) the two-point set $\{0,M/(1-q)\}$ for $1<q\leq q_G$; (ii) countably infinite for $q_G\leq q<q_{KL}$; (iii) uncountable but of zero Hausdorff dimension for $q=q_{KL}$; and (iv) of positive Hausdorff dimension for $q_{KL}<q\leq M+1$. The number $q_{KL}$ is called the {\em Komornik-Loreti constant}; see Section \ref{sec:symbolic} below for a precise definition. The above result was generalized to arbitrary $M\geq 1$ by Baker \cite{Baker} and Kong et al.~\cite{KLD}.

Let $\Omega:=\{0,1,\dots,M\}^\N$. The set $\u_q$ is most easily understood by studying the symbolic univoque set
\[
\us_q:=\{(x_i)\in\Omega: \pi_q((x_i))\in \u_q\}=\pi_q^{-1}(\u_q).
\]
For any subset $X\subseteq \Omega$, we define the {\em topological entropy} of $X$ by
\[
h(X):=\lim_{k\to\infty} \frac{\log\#B_k(X)}{k}=\inf_{k\in\N} \frac{\log \#B_k(X)}{k},
\]
where $B_k(X)$ is the set of all subwords of length $k$ which occur in some sequence in $X$, and $\#B$ denotes the cardinality of a set $B$. The above limit always exists, and is equal to the infimum, since it is easily seen that $\log\#B_k(X)$ is subadditive as a function of $k$. When $X$ is a subshift of the full shift $\Omega$, $h(X)$ coincides with the dynamical notion of topological entropy. 

In their paper, Glendinning and Sidorov suggested the formula
\begin{equation}
\dim_H \u_q=\frac{h(\us_q)}{\log q}
\label{eq:Hausdorff-dimension-formula}
\end{equation}
for the Hausdorff dimension of $\u_q$, and stated without proof that $h(\us_q)$ is continuous in $q$. Detailed proofs of these statements were recently given by Komornik et al.~\cite{Komornik-Kong-Li-2017}. Unfortunately, as we explain in Remark \ref{rem:error} below, their proof contains a serious error, which is not easily fixed. Since the publication of \cite{Komornik-Kong-Li-2017}, a number of papers (e.g. \cite{Allaart-2016,Allaart-2017a,Allaart-2017b,Allaart_Baker_Kong_2017}) have used the continuity of $\dim_H \u_q$ in a fundamental way, and it is therefore of crucial importance to have a complete proof of this result on record. Giving such a proof is the principal objective of this paper. We state our main results as follows:

\begin{theorem} \label{thm:main1}
The function $q\mapsto h(\us_q)$ is continuous on $(1,M+1]$.
\end{theorem}

\begin{theorem} \label{thm:main2}
For each $q\in(1,M+1]$, the formula \eqref{eq:Hausdorff-dimension-formula} holds.
\end{theorem}

Our initial approach is the same as in \cite{Komornik-Kong-Li-2017}: we sandwich the set $\us_q$ between two sets $\us_{q,n}$ and $\vs_{q,n}$ and show that $h(\us_{q,n})-h(\vs_{q,n})\to 0$ as $n\to\infty$. But, whereas the authors of \cite{Komornik-Kong-Li-2017} attempted to use the Perron-Frobenius theorem to compare the entropies of $\us_{q,n}$ and $\vs_{q,n}$, we give instead a more direct combinatorial argument by constructing for each $k\in\N$ a map $f_{n,k}$ from $B_k(\vs_{q,n})$ into $B_k(\us_{q,n})$ that is ``not too many"-to-one; see Section \ref{s:left continuity} for the details. We observe that we use some results from \cite{AlcarazBarrera-Baker-Kong-2016}, a paper which supercedes \cite{Komornik-Kong-Li-2017}. However, we emphasize that the continuity of $h(\us_q)$ is not used in \cite{AlcarazBarrera-Baker-Kong-2016}.

Theorem \ref{thm:main2} is a fairly direct consequence of Theorem \ref{thm:main1}, and is proved in Section \ref{sec:proof2}.

\section{Symbolic univoque sets} \label{sec:symbolic}

In this section we will describe the symbolic univoque set $\us_q$ and calculate its Hausdorff dimension. Let $\si$ be the \emph{left shift} on $\Omega$ defined by $\si((c_i))=(c_{i+1})$. Then $(\Omega,\si)$ is a \emph{full shift}. By a \emph{word} $\c$ we mean a finite string of digits $\c=c_1\ldots c_n$ with each digit $c_i\in {\{0,1,\dots,M\}}$. For two words $\c=c_1\ldots c_m$ and $\d=d_1\ldots d_n$ we denote by $\c\d=c_1\ldots c_m d_1\ldots d_n$ their concatenation. For a positive integer $n$ we write $\c^n=\c\cdots\c$ for the $n$-fold concatenation of $\c$ with itself. Furthermore, we write $\c^\f=\c\c\cdots$ for the infinite periodic sequence with period block $\c$. For a word $\c=c_1\ldots c_m$ we set $\c^+:=c_1\ldots c_{m-1}(c_m+1)$ if $c_m<M$, and set $\c^-:=c_1\ldots c_{m-1}(c_m-1)$ if $c_m>0$.
 Furthermore, we define the \emph{reflection} of the word $\c$ by $\overline{\c}:=(M-c_1)(M-c_2)\cdots(M-c_m)$. Clearly, $\c^+, \c^-$ and $\overline{\c}$ are all words with digits from ${\{0,1,\dots,M\}}$. For a sequence $(c_i)\in\Omega$ its reflection is also a sequence in $\Omega$ defined by $\overline{(c_i)}=(M-c_1)(M-c_2)\cdots$.

For a subset $X\subseteq \Omega$, the {\em language} of $X$, denoted $\L(X)$, is the set of all finite words that occur in some sequence in $X$. So, $\L(X)=\bigcup_{k=1}^\f B_k(X)$.

Throughout the paper we will use the lexicographical ordering $\prec, \lle, \succ$ and $\lge$ between sequences and words. More precisely, for two sequences $(c_i), (d_i)\in\Omega$ we say $(c_i)\prec (d_i)$ or $(d_i)\succ (c_i)$ if there exists an integer $n\ge 1$ such that $c_1\ldots c_{n-1}=d_1\ldots d_{n-1}$ and $c_n<d_n$. Furthermore, we write $(c_i)\lle (d_i)$ if $(c_i)\prec (d_i)$ or $(c_i)=(d_i)$. Similarly, for two words $\c$ and $\d$ we say $\c\prec \d$ or $\d\succ\c$ if $\c0^\f\prec \d0^\f$.

Let $q\in(1,M+1]$. The description of $\us_q$ is based on the \emph{quasi-greedy} $q$-expansion of $1$, denoted by $\al(q)=\al_1(q)\al_2(q)\ldots$,  which is the lexicographically largest $q$-expansion of $1$ not ending with $0^\f$ (cf.~\cite{Daroczy_Katai_1993}). The following characterization of $\al(q)$ was given in \cite[Proposition 2.3]{Vries-Komornik-Loreti-2016}.

\begin{lemma} \label{l21} 
The map $q\mapsto \al(q)$ is a strictly increasing bijection from $(1, M+1]$ onto the set of all sequences $(a_i)\in\Omega$ not ending with $0^\f$ and satisfying
\[
a_{n+1}a_{n+2}\ldots \lle a_1a_2\ldots \qquad\textrm{for all }\quad n\ge 0.
\]
Furthermore, the map $q\mapsto \al(q)$ is left-continuous.
\end{lemma}

The following lexicographic characterization of the symbolic univoque set $\us_q$ was essentially established by Parry \cite{Parry_1960} (see also \cite{Komornik-Kong-Li-2017}).

\begin{lemma} \label{l23}
Let $q\in(1,M+1]$. Then $(x_i)\in\us_q$ if and only if
\[\left\{\begin{array}{lll}
x_{n+1}x_{n+2}\ldots \prec \al(q)&\quad\textrm{whenever}& x_n<M,\\
x_{n+1}x_{n+2}\ldots \succ \overline{\al(q)}&\quad\textrm{whenever}& x_n>0.
\end{array}\right.
\]
\end{lemma}

By Lemmas \ref{l21} and \ref{l23} it follows that the set-valued map $q\mapsto \us_q$ is increasing, i.e., $\us_p\subseteq\us_q$ when $p<q$.

Next, we recall from \cite{Komornik_Loreti_2002} the definition of the Komornik-Loreti constant $q_{KL}=q_{KL}(M)$. 
Let $(\tau_i)_{i=0}^\f=0110100110010110\ldots$ be the classical Thue-Morse sequence (cf.~\cite{Allouche_Shallit_1999}). Then $q_{KL}$ is given implicitly by
\begin{equation*} %\label{eq:21}
\al(q_{KL})=\la_1\la_2\ldots,
\end{equation*}
where for each $i\ge 1$,
\begin{equation*} %\label{eq:22}
\la_i=\la_i(M):=\begin{cases}
k+\tau_i-\tau_{i-1} & \qquad\textrm{if $M=2k$},\\
k+\tau_i & \qquad\textrm{if $M=2k+1$}.
\end{cases}
\end{equation*}
For example, $q_{KL}(1)\approx 1.78723$, $q_{KL}(2)\approx 2.53595$, $q_{KL}(3)\approx 2.91002$, etc.

We shall also need the set
\[
\ub:=\{q\in(1,M+1]: 1\in \u_q\}.
\]
This set is of Lebesgue measure zero, and $\min \ub=q_{KL}$ (see \cite{Komornik_Loreti_2002}). The following characterizations of $\ub$ and its topological closure $\overline{\ub}$ were established in \cite{Komornik_Loreti_2007} (see also \cite{Vries-Komornik-Loreti-2016}).

\begin{lemma}
\label{l34}
\mbox{}

\begin{enumerate}[{\rm(i)}]
\item $q\in\ub$ if and only if 
$
\overline{\al(q)}\prec\si^n(\al(q))\prec \al(q)$ {for all} $n\ge 1.$

\item $q\in\overline{\ub}$ if and only if 
 $
\overline{\al(q)}\prec\si^n(\al(q))\lle \al(q)$ {for all} $n\ge 1.$
\end{enumerate}
\end{lemma}

The following definition was taken from \cite[Definition 3.10]{AlcarazBarrera-Baker-Kong-2016}.

\begin{definition} \label{def:admissible}
Say a word $a_1\dots a_m$ in $\L(\Omega)$ is {\em primitive} if
\[
\overline{a_1\ldots a_{m-i}} \prec a_{i+1}\ldots a_m \lle a_1\ldots a_{m-i}\qquad\textrm{for all}\quad {0}\le i<m.
\]
\end{definition}

For a proof of the next Lemma, see \cite[Lemma 4.1]{Komornik_Loreti_2007}.

\begin{lemma} \label{lem:admissible-words}
Let $q\in\overline{\ub}$. %and write {\color{blue}$\alpha(q)=a_1 a_2\dots$}. 
Then there are infinitely many positive integers $n$ such that $\alpha_1(q)\dots\alpha_n(q)$ is primitive.
\end{lemma}

Note that $\us_q$ is in general not a subshift of $\Omega$. Following \cite{DeVries_Komornik_2008} and \cite{Komornik-Kong-Li-2017}
we introduce the set
\begin{equation*}
\vs_q:=\set{(x_i)\in\Omega: \overline{\al(q)}\lle x_{n+1}x_{n+2}\ldots \lle \al(q)\textrm{ for all }n\ge 0}.
\end{equation*}
Then $\vs_q$ is a subshift of $\Omega$.
Comparison of this definition with the characterization of $\us_q$ in Lemma \ref{l23} suggests that $\us_q$ and $\vs_q$ should have the same entropy. This is indeed the case:

\begin{proposition} \label{prop:same-entropy}
For all $q\in(q_{KL},M+1)$, $h(\us_q)=h(\vs_q)$.
\end{proposition}

\begin{proof}
We first show that $h(\us_q)\leq h(\vs_q)$.
By Lemma \ref{l23} it follows that for each $q\in(1, M+1]$ the set $\us_q$ is contained in a countable union of affine copies of $\vs_q$ (see also \cite[Lemma 3.2]{Kalle-Kong-Li-Lv-2016}), i.e., there exists a sequence of affine maps $\set{g_i}_{i=1}^\f$ on $\Omega$ of the form
\[
x_1x_2\ldots\mapsto a x_1x_2\ldots,\qquad x_1x_2\ldots \mapsto M^m b x_1x_2\ldots \qquad\mbox{or}\qquad x_1x_2\ldots \mapsto 0^m c x_1x_2\ldots,
\]
where $a\in\set{1,2,\ldots, M-1}$, $b\in\set{0,1,\ldots, M-1}$, $c\in\set{1,2,\ldots, M}$ and $m=1,2,\ldots$, such that
\begin{equation}
\us_q\subseteq\bigcup_{i=1}^\f g_i(\vs_q).
\label{eq:26}
\end{equation}
Hence any word in $\L(\us_q)$ of length $k$ is either itself in $\L(\vs_q)$, or else has a prefix $a$, $M^m b$ or $0^m c$ followed by a word in $\L(\vs_q)$, with $a,b$ and $c$ as above. Thus, %{\color{blue}by (\ref{eq:26})}
\begin{align*}
\#B_k(\us_q) &\leq \#B_k(\vs_q)+(M-1)\#B_{k-1}(\vs_q)+2M\sum_{m=0}^{k-1}\#B_{k-m-1}(\vs_q)\\
&\leq (2k+1)M\#B_k(\vs_q)
\end{align*}
for all $k$, and hence $h(\us_q)\leq h(\vs_q)$.

For the reverse inequality, note that $\vs_q\backslash \us_q$ contains only sequences ending in $\alpha(q)$ or $\overline{\alpha(q)}$. Hence $\vs_q\backslash \us_q$ is countable. This does not immediately imply that $h(\vs_q)\leq h(\us_q)$, since the topological entropy of a countable set can be positive. To verify the inequality rigorously, it suffices in view of Lemma \ref{l34} to consider the following three cases:

\medskip
{\em Case 1:} $q\in\overline{\ub}$. We show that in this case,
\begin{equation}
B_k(\vs_q)\subseteq B_k(\us_q) \qquad\mbox{for all $k\in\N$}.
\label{eq:word-inclusion}
\end{equation}
Let $x_1\dots x_k\in B_k(\vs_q)$. Then there is a sequence $(y_i)\in \vs_q$ and an index $j\in\N$ such that $y_{j+1}\dots y_{j+k}=x_1\dots x_k$. Since $(y_i)\in \vs_q$, we have
\[
\overline{\alpha(q)}\preceq y_{i+1}y_{i+2}\dots \preceq \alpha(q) \qquad\mbox{for all $i\geq 0$}.
\]
Suppose without loss of generality that $y_{i+1}y_{i+2}\dots=\alpha(q)$ for some $i$. Write $\alpha(q)=\alpha_1\alpha_2\dots$.
Since $q\in \overline{\ub}$, there is by Lemma \ref{lem:admissible-words} an index $n>\max\{i,j+k\}$ such that $\alpha_1\dots \alpha_n$ is primitive. Consider the sequence
\begin{align*}
z_1 z_2\dots &=y_1\dots y_j x_1\dots x_k y_{j+k+1}\dots y_{n+i}^-(\alpha_1\dots\alpha_n^-)^\infty\\
&=y_1\dots y_i (\alpha_1\dots\alpha_n^-)^\infty.
\end{align*}
Then $z_1 z_2\dots \in \us_q$. Hence $x_1\dots x_k\in B_k(\us_q)$, proving \eqref{eq:word-inclusion}.

\medskip
{\em Case 2:} $\sigma^n(\alpha(q))\prec\overline{\alpha(q)}$ for some $n\geq 1$. Then it is not possible for a sequence in $\vs_q$ to end in $\alpha(q)$ or $\overline{\alpha(q)}$ in view of the definition of $\vs_q$. Hence, $\vs_q\subseteq\us_q$.

\medskip
{\em Case 3:} $\sigma^n(\alpha(q))=\overline{\alpha(q)}$ for some $n\geq 1$. This means that, with $v:=\alpha_1\dots\alpha_n$, $\alpha(q)=(v\overline{v})^\infty$. So, if $(x_i)\in \vs_q$ and $x_{j+1}\dots x_{j+n}=v$, it must be the case that $\sigma^j((x_i))=(v\overline{v})^\infty$; and likewise, if 
$x_{j+1}\dots x_{j+n}=\overline{v}$ then $\sigma^j((x_i))=(\overline{v}v)^\infty$. Thus, any word in $\L(\vs_q)$ consists of a word in $\L(\us_q)$ followed by $v$ or $\overline{v}$, followed in turn by a forced suffix. As a result,
\[
\#B_k(\vs_q)\leq 2\sum_{j=0}^k \#B_j(\us_q)\leq 2(k+1)\#B_k(\us_q)
\]
for all $k\in\N$, and therefore, $h(\vs_q)\leq h(\us_q)$.
\end{proof}

As a final preparation for the proof of Theorem \ref{thm:main1}, we define for each $n\in\N$ the sets
\[
\us_{q,n}:=\{(x_i)\in\Omega: \overline{a_1\dots a_n}\prec x_{j+1}\dots x_{j+n}\prec a_1\dots a_n\ \mbox{for all $j\geq 0$}\}
\]
and
\[
\vs_{q,n}:=\{(x_i)\in\Omega: \overline{a_1\dots a_n}\preceq x_{j+1}\dots x_{j+n}\preceq a_1\dots a_n\ \mbox{for all $j\geq 0$}\},
\]
where we write $\alpha(q)=a_1 a_2\dots$.
Then $(\us_{q,n}, \si)$ and $(\vs_{q,n}, \si)$ are both subshifts of finite type for any $n\ge 1$. 
Observe from \cite[Lemma 2.7 ]{Komornik-Kong-Li-2017} that 
\begin{equation*} 
\us_{q,n}\subseteq \vs_q\subseteq\vs_{q,n}\quad\textrm{for all }n\ge 1.
\end{equation*}
Furthermore, the set sequence  $(\us_{q,n})$ is nondecreasing and the set sequence $(\vs_{q,n})$ is nonincreasing.  

\begin{remark}[The error in the proof of Komornik, Kong and Li] \label{rem:error}
The authors of \cite{Komornik-Kong-Li-2017} applied the Perron-Frobenius theorem to the edge graph representation $G(n)$ of $\us_{q,n}$ to obtain constants $c_1$ and $c_2$ such that
\[
c_1\lambda_n^k\leq \#B_k(\us_{q,n})\leq c_2 k^s \lambda_n^k,
\]
where $\lambda_n$ is the spectral radius, and $s$ the number of strongly connected components, of $G(n)$. However, later in their proof they treat $c_1$ and $c_2$ as absolute constants, whereas in fact they depend on $n$. The method of proof in \cite{Komornik-Kong-Li-2017} could be saved by finding good bounds on the growth rate of $c_2(n)$, but this turns out to be very difficult to do. Despite our best efforts, we have not been able to accomplish this; hence our resort to the combinatorial method of Section \ref{s:left continuity} below.
\end{remark}

Let $H:(1,M+1]\to [0,\infty)$ be the map
\[
H(q):=h(\us_q).
\]

The right continuity of $H$ is easy to prove: 

\begin{proposition} \label{prop:right-continuity}
The function $H$ is right continuous on $(1,M+1]$.
\end{proposition}

\begin{proof}
By [21, Theorem 2.6] it follows that $H$ is constant on each connected component of $(1,M+1]\setminus\ub=(1, q_{KL})\cup\bigcup[q_0, q_0^*)$. So, we only need to prove the right continuity of $H$ on $\ub$. Take $q\in\ub$, then by Lemma \ref{l34} (i) $\al(q)=\beta(q)$, where $\beta(q)$ is the greedy $q$-expansion of $1$. We first show that
\begin{equation}
\lim_{n\to\infty} h(\vs_{q,n})\leq h(\vs_q).
\label{eq:v-entropy-inequality} 
\end{equation}
By \cite[Lemma 2.10]{Komornik-Kong-Li-2017}, $B_n(\vs_{q,n})=B_n(\vs_q)$ for each $n$. Hence
\[
h(\vs_{q,n})=\inf_{k\in\N}\frac{\log\#B_k(\vs_{q,n})}{k}\leq \frac{\log\#B_n(\vs_{q,n})}{n}=\frac{\log\#B_n(\vs_{q})}{n},
\]
and letting $n\to\infty$ gives \eqref{eq:v-entropy-inequality}.

Next, we can choose for each $n\in\N$ a base $p_n\in(q,M+1)$ sufficiently close to $q$ such that
$\alpha_i(p_n)=\alpha_i(q)$ for $i=1,\dots,n$. Since $\alpha(q)=\beta(q)$, we have
\[
\vs_q\subseteq \vs_p\subseteq \vs_{q,n} \qquad \mbox{for all $p\in(q,p_n)$}.
\]
It follows by \eqref{eq:v-entropy-inequality} that $\lim_{p\searrow q} h(\vs_p)=h(\vs_q)$, and then also $\lim_{p\searrow q} h(\us_p)=h(\us_q)$ in view of Proposition \ref{prop:same-entropy}.
\end{proof}

The proof of left continuity of $H$ is much more involved, and the next section is entirely devoted to this task.

\section{Left continuity of $H$}\label{s:left continuity}

%In this section we will prove the left continuity of $H$. To this end, l
Let $\bb$ be the \emph{bifurcation set} of the entropy function $H$, defined by
\[
\bb=\bb(M):=\set{q\in(1,M+1]: H(p)\ne H(q)\textrm{ for any }p\ne q}.
\]
Alcaraz Barrera et al.~\cite{AlcarazBarrera-Baker-Kong-2016} proved that $\bb\subseteq\ub$, and hence $\bb$ is of zero Lebesgue measure. They also showed that $\bb$ has full Hausdorff dimension.  Furthermore, $\bb$ has no isolated points and its complement can be written as
\begin{equation*} %\label{e15}
(1, M+1]\setminus\bb=(1,q_{KL}]\cup\bigcup[p_L, p_R],
\end{equation*}
where the union on the right hand side is countable and pairwise disjoint. From the definition of $\bb$ it follows that each connected component $[p_L, p_R]$ is a maximal interval on which $H$ is constant. Each closed interval $[p_L,p_R]$ is therefore called an \emph{entropy plateau}. 
Observe that $H$ is trivially left continuous on each half open connected component $(p_L, p_R]$ (including the first connected component $(1, q_{KL}]$). Hence it suffices to prove the left-continuity of $H$ at any 
\[
q\in(q_{KL}, M+1]\setminus\bigcup(p_L, p_R]=:\bb^L.
\] 
(The ``left bifurcation set" $\bb^L$ was introduced and studied in \cite{Allaart_Baker_Kong_2017}.) 
   
\begin{theorem} \label{th:left-continuity}
For any $q\in\bb^L$ we have 
\begin{equation}
\lim_{n\ra\f}h(\vs_{q,n})=\lim_{n\ra\f}h(\us_{q,n}).
\label{eq:same-entropy-limits}
\end{equation}
\end{theorem}

\begin{corollary} \label{cor:left-continuity}
The function $H$ is left continuous on $\bb^L$.
\end{corollary}

Before proving Theorem \ref{th:left-continuity}, we show how to derive the corollary.

\begin{proof}[Proof of Corollary \ref{cor:left-continuity}]
Fix $q\in\bb^L$. For each $n$ we can choose a base $p_n\in(1,q)$ sufficiently close to $q$ so that $\alpha_i(p_n)=\alpha_i(q)$ for $i=1,\dots,n$. Then
\[
\us_{q,n}\subseteq \vs_p \subseteq \vs_q \subseteq \vs_{q,n} \qquad\mbox{for all $p\in(p_n,q)$}.
\]
By \eqref{eq:same-entropy-limits} and the above inclusions, $\lim_{n\to\infty} h(\us_{q,n})=h(\vs_q)$. Hence, $\lim_{p\nearrow q} h(\vs_p)=h(\vs_q)$, and then also $\lim_{p\nearrow q} h(\us_p)=h(\us_q)$ in view of Proposition \ref{prop:same-entropy}.
\end{proof}

Our approach to proving Theorem \ref{th:left-continuity} is to construct, for arbitrarily large numbers $n$ and for all $k\in\N$, a map $f_{n,k}: B_k(\vs_{q,n}) \to B_k(\us_{q,n})$ that is ``not too many"-to-one. (We will specify later what ``not too many" means.) This will show that the set $B_k(\vs_{q,n})$ is not too much larger than $B_k(\us_{q,n})$, and as a consequence, $h(\vs_{q,n})$ is not too much larger than $h(\us_{q,n})$.

Recall the definition of a primitive word from Definition \ref{def:admissible}.   
\begin{lemma}\label{lem:22}
Let $q\in\bb^L$ with $\al(q)=(a_i)$. 
\begin{enumerate}[{\rm(i)}]
\item There exist infinitely many integers $n$ such that 
\begin{equation}
a_1\ldots a_n\textrm{ is primitive}\quad \textrm{and}\quad (a_1\ldots a_n^-)^\f\succ\al(q_{KL}).
%a_1\ldots a_n(\overline{a_1\ldots a_n}^+)^\f\prec (a_i).
\label{eq:admissible-and-more}
\end{equation}
\item If \eqref{eq:admissible-and-more} holds, then $a_1\ldots a_n(\overline{a_1\ldots a_n}^+)^\f\prec (a_i)$.
%$a_1\ldots a_n^-$ is admissible and $(a_1\ldots a_n^-)^\f\succ\al(q_{KL})$, then 
%\[a_1\ldots a_n(\overline{a_1\ldots a_n}^+)^\f\prec (a_i).\]
\end{enumerate}
\end{lemma}
	
\begin{proof}
Since $\bb^L\subseteq\overline{\ub}$, by Lemma 2.5 there are infinitely many $n\in\N$ such that $a_1\ldots a_n$ is primitive. Furthermore, for all large enough $n$, $(a_1\ldots a_n^-)^\f\succ\al(q_{KL})$ since $q>q_{KL}$. So, it suffices to show that for such a large integer $n$ we have $a_1\ldots a_n(\overline{a_1\ldots a_n}^+)^\f\prec \al(q)$.

Take such a large integer $n$, and let $[q_L, q_R]$ be the interval determined by
\[
\al(q_L)=(a_1\ldots a_n^-)^\f\qquad\textrm{and}\qquad \al(q_R)=a_1\ldots a_n(\overline{a_1\ldots a_n}^+)^\f.
\]
By a similar argument as in the proofs of \cite[Lemmas 5.1 and 5.5]{AlcarazBarrera-Baker-Kong-2016} one can show that $H$ is constant on $[q_L, q_R]$. Since $q>q_L$, by the definition of $\bb^L$ it follows that $q>q_R$, and hence $a_1\ldots a_n(\overline{a_1\ldots a_n}^+)^\f\prec \al(q)$ by Lemma \ref{l21}.
\end{proof}

\begin{definition}\label{def:n}
For $q\in\bb^L$ with $\alpha(q)=(a_i)$, let 
\[
\mathcal{N}(q):=\{n\in\N: a_1\dots a_n\textrm{ is primitive and }(a_1\ldots a_n^-)^\f\succ\al(q_{KL})\}. 
\]
\end{definition}

Take $q\in\bb^L$ and fix $m\in\mathcal N(q)$.  By Lemma \ref{lem:22} (ii) it follows that    
$a_1\ldots a_m(\overline{a_1\ldots a_m}^+)^\f\prec \al(q).$
So, there exist integers $l=l(m)\ge 0$ and $r=r(m)\in\set{1,\ldots, m}$ such that for $n=n(m):=m(l+1)+r$,
  \begin{equation} \label{eq:kong-1}
  \begin{split}
& a_1\ldots a_{n-1}=a_1\ldots a_m(\overline{a_1\ldots a_m}^+)^l\overline{a_1\ldots a_{r-1}},\\
  & a_{n}>\overline{a_r}\qquad\textrm{if}\quad r<m,\\
  & a_{n}>\overline{a_m}^+\quad\textrm{if}\quad r=m.
  \end{split}
  \end{equation}
  We point out that the integers $l, r$ and $n$ all depend on $m$ (and, of course, on $q$). However, most of the time the base $q\in\bb^L$ is fixed, and if the value of $m$ is implicitly understood we will write $l, r$ and $n$ instead of $l(m), r(m)$ and $n(m)$.

  \begin{lemma} \label{lem:23} 
	Let $q\in\bb^L$ with $\al(q)=(a_i)$. Suppose $m\in\mathcal N(q)$ and $n=m(l+1)+r$ as in (\ref{eq:kong-1}). The following statements hold.
  \begin{enumerate}
  \item[{\rm{(i)}}] $a_{n-m+1}\ldots a_{n}^-\succ\overline{a_1\ldots a_m}$.
  
  \item[{\rm{(ii)}}] $n\in\mathcal N(q)$. Thus, $a_1\ldots a_n$ is primitive and $a_1\ldots a_n(\overline{a_1\ldots a_n}^+)^\f\prec\al(q)$. 
  \item[{\rm{(iii)}}] $a_1\ldots a_r$ is primitive. 
  \end{enumerate}
  \end{lemma}
  
  \begin{proof}
  First we prove (i). From (\ref{eq:kong-1}) we see that if $r=m$, then $a_{n-m+1}\ldots a_n^-\lge \overline{a_1\ldots a_m}^+\succ \overline{a_1\ldots a_m}$. If $r<m$ with $l>0$, then $n-m=ml+r$, and it follows by primitivity of $a_1\ldots a_m$ that 
  \[a_{n-m+1}\ldots a_{n-r}=\overline{a_{r+1}\ldots a_m}^+\succ\overline{a_1\ldots a_{m-r}},\]
  which implies $a_{n-m+1}\ldots a_n^-\succ \overline{a_1\ldots a_m}$. Furthermore, if $r<m$ with $l=0$, then we deduce from the primitivity of $a_1\ldots a_m$ that 
  \[a_{n-m+1}\ldots a_{n-r}=a_{r+1}\ldots a_m\succ \overline{a_1\ldots a_{m-r}}.\]
  Again this yields $a_{n-m+1}\ldots a_n^-\succ \overline{a_1\ldots a_m}$. So (i) holds. 
  
  For (ii), since $(a_1\ldots a_n^-)^\f\succ(a_1\ldots a_m^-)^\f\succ\al(q_{KL})$, by Lemma \ref{lem:22} (ii) it suffices to prove that $a_1\ldots a_n$ is primitive. 
	By Lemma \ref{l21},  
  \[
  a_{j+1}\ldots a_n^-\prec a_{j+1}\ldots a_n\lle a_1\ldots a_{n-j}\quad\textrm{for all }0\le j<n,
  \]
  so by Definition \ref{def:admissible} it suffices to prove that
  \begin{equation}\label{eq:hx-21}
  a_{j+1}\ldots a_n^-\lge \overline{a_1\ldots a_{n-j}}\quad\textrm{for all }0\le j<n.
  \end{equation}
  Note by (\ref{eq:kong-1}) that $n>m$, and $a_1\dots a_n^-$ begins with $a_1\dots a_m$.
		The primitivity of $a_1\ldots a_m$ gives 
  \begin{equation*}
  a_{j+1}\ldots a_m\succ \overline{a_1\ldots a_{m-j}} \quad \textrm{and}\quad \overline{a_{j+1}\ldots a_m}^+\succ \overline{a_1\ldots a_{m-j}}
  \end{equation*}
  for all $0\le j<m$. The first inequality gives \eqref{eq:hx-21} for $0\leq j<m$; the second inequality implies \eqref{eq:hx-21} for $m\leq j<n$, using \eqref{eq:kong-1}.
This establishes (ii). 
  
  Now we turn to prove (iii). If $r=m$, then (iii) follows trivially from the primitivity of $a_1\ldots a_m$. Suppose $r<m$. It suffices to prove 
  \[
  a_{i+1}\ldots a_r^-\lge \overline{a_1\ldots a_{r-i}}\quad\textrm{for all }0\le i<r.
  \]
  This follows from (ii) and (\ref{eq:kong-1}), since
  \[
  \overline{a_{i+1}\ldots a_r^-}\lle a_{n-r+i+1}\ldots a_n\lle a_1\ldots a_{r-i} 
  \] 
	for all $0\le i<r$.
  \end{proof}

 From now on, we will fix a base $q\in\bb^L$.

Fix a number $k>n$. We will construct a map $f_{n, k}: B_k(\vs_{q,n})\to B_k(\us_{q,n})$ and show that this map is ``not too many"-to-one. 
The map $f_{n,k}$ will be defined as the $k$th iterate of an auxiliary function $F_{\su,\sv}$; see Definition \ref{def:general-F} below. We also use the following notation:

\begin{definition}
For a primitive word $\su={a_1\dots a_n}$ and a word $\sx=x_1\dots x_k\in B_k(\vs_{q,n})$, let
\[
i_\su(\sx):=\min\{i\leq k-n: x_{i+1}\dots x_{i+n}=\su\ \mbox{or}\ x_{i+1}\dots x_{i+n}=\overline{\su}\},
\]
or $i_\su(\sx):=\infty$ if no such $i$ exists.
\end{definition}

Thus, $i_\su(\sx)$ indicates where the word $\su$ or $\overline{\su}$ occurs for the first time in the word $\sx$. Note that $i_\su(\sx)=\infty$  if and only if $\sx\in B_k(\us_{q,n})$.

\begin{definition} \label{def:general-F}
Let $\su=a_1\dots a_n$ be primitive, and let $\sv=a_1\dots a_m$ be a primitive prefix of $\su$. Write $\su={\sv\mathbf z}$. Then we define the map $F_{\su,\sv}: B_k(\vs_{q,n})\to\{0,1,\dots,M\}^k$ as follows:
\begin{enumerate}
 \item If $\sx=x_1\ldots x_k\in B_k(\vs_{q,n})$ does not contain the word $\su$ or $\overline{\su}$, then set $F_{\su,\sv}(\sx):=\sx$. 
 \item Otherwise, let $i:=i_\su(\sx)$. %the first occurrence of $\su$ or $\overline{\su}$ in $\sx$ begin at position $i+1$. 
 \begin{itemize}
 \item If $x_{i+1}\ldots x_{i+n}=\su$, then we put
   \begin{align*}
   F_{\su,\sv}(\sx)=F_{\su,\sv}(x_1\ldots x_i\sv{\mathbf z} x_{i+n+1}\ldots x_k):=x_1\ldots x_i\sv^-\overline{{\mathbf z} x_{i+n+1}\ldots x_k}.
   \end{align*}
 \item If $x_{i+1}\ldots x_{i+n}=\overline{\su}$, then we put 
 \[
  F_{\su,\sv}(\sx)=F_{\su,\sv}(x_1\ldots x_i\overline{\sv{\mathbf z}} x_{i+n+1}\ldots x_k):=x_1\ldots x_i \overline{\sv}^+ {\mathbf z} \overline{x_{i+n+1}\ldots x_k}.
 \]
\end{itemize}
\end{enumerate}
\end{definition}

From Definition \ref{def:general-F} it follows that 
\[F_{\su,\sv}(\overline{\sx})=\overline{F_{\su,\sv}(\sx)}\quad \textrm{for any }\sx\in B_k(\vs_{q,n}).\]

In each of the cases worked out below, the key is to choose $\su$ and $\sv$ carefully and show that $F_{\su,\sv}$ maps $B_k(\vs_{q,n})$ into itself, so that the $k$th iterate $F_{\su,\sv}^k$ is well defined and maps $B_k(\vs_{q,n})$ into $B_k(\us_{q,n})$.

\subsection{Construction of $f_{n,k}$: the first case.}
	
Assume first that $l(m)>0$ for infinitely many $m\in\mathcal N(q)$.
Take $q\in\bb^L$ and fix $m\in\mathcal N(q)$ with $l=l(m)>0$. Let $n=m(l+1)+r$ as in (\ref{eq:kong-1}).       Write 
  \[
	\su:=a_1\ldots a_n=\sv(\overline{\sv}^+)^l \sw,\quad\textrm{where}\quad \sv:=a_1\ldots a_m,\quad \sw:=a_{n-r+1}\ldots a_n.
	\]
With $\su$ and $\sv$ as above, we set $F:=F_{\su,\sv}$. The following lemma shows that $F$ is a map from $B_k(\vs_{q,n})$ to $B_k(\vs_{q,n})$.

 \begin{lemma} \label{lem:24}
 For any $\sx\in B_k(\vs_{q,n})$ we have $F(\sx)\in B_k(\vs_{q,n})$, and  
\[
i_\su(F(\sx))\geq i_\su(\sx)+\frac{n}{2}.
\]
 \end{lemma}

  \begin{proof}
	Let $i:=i_\su(\sx)$. By symmetry we may assume $x_{i+1}\ldots x_{i+n}=\su$, so 
	\[
	\sx=x_1\ldots x_i\sv(\overline{\sv}^+)^l\sw x_{i+n+1}\ldots x_k.
	\]
  Let $F(\sx)=y_1\ldots y_k$. By Definition \ref{def:general-F} it follows that 
  \begin{equation} \label{eq:yixi}
  y_1\ldots y_k=x_1\ldots x_i(\sv^-)^{l+1}\overline{\sw x_{i+n+1}\ldots x_k}.
  \end{equation}
   Since the entire word $x_{i+m+1}\ldots x_k$ is being reflected by $F$, no word strictly greater than $\su$ or strictly smaller than $\overline{\su}$ can occur in $y_{i+m+1}\ldots y_k$. To prove the lemma, therefore, it is necessary and sufficient to show that for each $j<i+(n/2)$,
  \[
  \overline{\su}\prec y_{j+1}\ldots y_{j+n}\prec \su.
  \]
  
 Note by \eqref{eq:yixi} that $y_1\ldots y_{i+m}=x_1\ldots x_{i+m}^-$.  By the minimality of $i$ it follows that 
  $y_{j+1}\ldots y_{j+n}=x_{j+1}\ldots x_{j+n}\prec \su$  {for all }$0\le j<i+m-n.$ 
  Furthermore, for $i+m-n\le j<i+m$ we have $y_{j+1}\ldots y_{j+n}\prec x_{j+1}\ldots x_{j+n}\lle \su$. So, $y_{j+1}\ldots y_{j+n}\prec\su$ for all $j<i+m$. And for $i+m\leq j<i+(n/2)$, we have $j+m<i+n$ since $n>2m$. Then the same inequality follows since $\sv$ is primitive.

  Proving the other inequality,
  \begin{equation}\label{eq:kong-2}
  y_{j+1}\ldots y_{j+n}\succ\overline{\su}\quad\textrm{for all }j<i+\frac{n}{2},
  \end{equation}
  is more involved. First, by the minimality of $i$ it follows that 
  \[
  y_{j+1}\ldots y_{j+n}=x_{j+1}\ldots x_{j+n}\succ \overline{\su}\quad\textrm{for all }0\le j<i+m-n.
  \] 
So it remains to prove (\ref{eq:kong-2}) for $j\geq i+m-n$. We consider {four} cases (see Figure \ref{Fig1}):

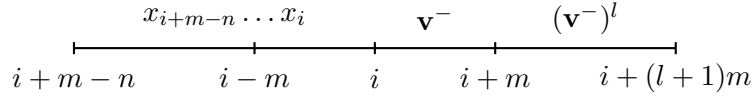
\begin{figure}[h!]
\begin{center}
\begin{tikzpicture}[ 
    scale=8,
    axis/.style={very thick, ->},
    important line/.style={thick},
    dashed line/.style={dashed, thin},
    pile/.style={thick, ->, >=stealth', shorten <=2pt, shorten
    >=2pt},
    every node/.style={color=black}
    ]
    % axis
 
    % Lines
 
    \draw[important line]  (0,0)--(1,0);
                \draw[important line] (0,-0.01)--(0, 0.01);
                 \draw[important line] (0.3,-0.01)--(0.3, 0.01);
                           \draw[important line] (0.5,-0.01)--(0.5, 0.01);  
                                     \draw[important line] (0.7,-0.01)--(0.7, 0.01);  
                                               \draw[important line] (1,-0.01)--(1, 0.01);  
      \node[] at (0,-0.05){$i+m-n$};
        \node[] at (0.3, -0.05){$i-m$};
         \node[] at (0.5, -0.05){$i$};
          \node[] at (0.7, -0.05){$i+m$};
           \node[] at (1, -0.05){$i+(l+1)m$};
        
           \node[] at (0.25, 0.05){$x_{i+m-n}\ldots x_i$};
                  \node[] at (0.6, 0.05){$\sv^-$};
             \node[] at (0.85, 0.05){$(\sv^-)^l$};
            
   \end{tikzpicture} 
\end{center}
\caption{The presentation of $y_{i+m-n}\ldots y_{i+(l+1)m}=x_{i+m-n}\ldots x_i(\sv^-)^{l+1}$.}
\label{Fig1}
\end{figure}

\begin{itemize}
 \item[(I).] $j=i+m-n$. Then $y_{j+1}\ldots y_{j+n}=x_{j+1}\ldots x_{i+m}^-=x_{j+1}\ldots x_{i}a_1\ldots a_m^-$. Note that $ y_{j+1}\ldots y_i=x_{j+1}\ldots x_{i}\lge \overline{a_1\ldots a_{i-j}}=\overline{a_1\ldots a_{n-m}}.
$
 Furthermore, by Lemma \ref{lem:23} (i) it follows that 
 \begin{equation*}
 y_{i+1}\ldots y_{j+n}=a_1\ldots a_m^-\succ\overline{a_{n-m+1}\ldots a_n}.
 \end{equation*}
 This proves (\ref{eq:kong-2}) for $j=i+m-n$.
  
  \item[(II).] $i+m-n<j\le i-m$. Then $n>2m$. %So this case can only happen when $l>0$.  
	Write $j+n=i+t m+s$ with $t\in\{1,\dots,l\}$ and $s\in\set{1,\ldots, m}$. Then it follows from (\ref{eq:yixi}) that   
	  \begin{align}
  y_{j+1}\ldots y_i&=x_{j+1}\ldots x_i, \notag \\
	y_{i+1}\ldots y_{i+tm}&=(a_1\ldots a_m^-)^t, \label{eq:repeating-part}
  \end{align}
	and
  \begin{equation*}
  y_{i+tm+1}\ldots y_{j+n}=\begin{cases}
	a_1\ldots a_s & \textrm{if }s<m,\\
  a_1\ldots a_m^- &\textrm{if }s=m.
	\end{cases}
  \end{equation*}

 Note that $y_{j+1}\ldots y_i=x_{j+1}\ldots x_i\lge \overline{a_1\ldots a_{i-j}}$.  Observe that  $\su=a_1\ldots a_n=a_1\ldots a_m(\overline{a_1\ldots a_m}^+)^l\sw$ and $m\le i-j<n-m$. Then by \eqref{eq:kong-1} $\overline{a_{i-j+1}\ldots a_{n-s}}$ is a subword of $(a_1\ldots a_m^-)^l$, so using \eqref{eq:repeating-part} and the primitivity of $a_1\ldots a_m$ it follows that 
 \begin{equation}\label{eq:hx-22}
y_{j+1}\ldots y_{i+tm}=x_{j+1}\ldots x_i(a_1\ldots a_m^-)^t\lge \overline{a_{1}\ldots a_{i-j+tm}}=\overline{a_1\ldots a_{n-s}}.
 \end{equation}
 Since by Lemma \ref{lem:23} (ii) $\su=a_1\ldots a_n$ is primitive, we also have that when $s<m$,
 \[ y_{i+tm+1}\ldots y_{j+n}= a_1\ldots a_s\succ \overline{a_{n-s+1}\ldots a_n}.\]
 On the other hand, when $s=m$ we have by Lemma \ref{lem:23} (i) that 
 \[
 y_{i+tm+1}\ldots y_{j+n}=a_1\ldots a_m^-\succ  \overline{a_{n-m+1}\ldots a_n}. 
 \]
Combining this with (\ref{eq:hx-22}) gives $y_{j+1}\ldots y_{j+n}\succ \overline{\su}$. This proves (\ref{eq:kong-2}) for $i+m-n<j\le i-m$.
 
  \item[(III).] $i-m<j<i$. Then $y_{j+1}\ldots y_i=x_{j+1}\ldots x_i\lge \overline{a_1\ldots a_{i-j}}$. Note that $i-j<m$. Then by the primitivity of $a_1\ldots a_m$ it follows that 
  \[
  y_{i+1}\ldots y_{j+m}=a_1\ldots a_{j+m-i}\succ\overline{a_{i-j+1}\ldots a_m}.
  \]
  This proves (\ref{eq:kong-2}) for $i-m<j<i$.
 
   \item[(IV).] {$i\le j<i+(n/2)$.  Recall that we are assuming $l>0$. Then $j+m<i+n$.  By (\ref{eq:yixi}) we have  
   \[y_{j+1}\ldots y_{j+m}=a_{t+1}\ldots a_m^- a_1\ldots a_{t}\quad\textrm{for some }0\le t<m.\]
     By the primitivity of $a_1\ldots a_m$ it follows that 
  $y_{j+1}\ldots y_{j+m}\succ\overline{a_1\ldots a_m}$. %\quad\textrm{for all }i\le j<i+m.
  Hence (\ref{eq:kong-2}) holds for $i\le j<i+(n/2)$.}

    \end{itemize}
  We have now shown (\ref{eq:kong-2}) for all $j<i+(n/2)$. The proof is complete.
  \end{proof}

  As a result of Lemma \ref{lem:24}, for some large enough $j$ (with $j<k$) we have $F^k(\sx)=\dots=F^{j+1}(\sx)=F^j(\sx)$. 
	
	\begin{definition} \label{def:fnk}
	We define  
	\[
	f_{n,k}(\sx):=F^k(\sx), \qquad \sx\in B_k(\vs_{q,n}).
	\]
	\end{definition}

Observe that $F(f_{n,k}(\sx))=f_{n,k}(\sx)$, so $f_{n,k}(\sx)$ does not contain the word $\su$ or $\overline{\su}$. Hence, $f_{n,k}$ maps $B_k(\vs_{q,n})$ into $B_k(\us_{q,n})$. 
  
\begin{proposition}\label{prop:left-continuity-positive}
Let $q\in\bb^L$. If $l(m)>0$ for infinitely many $m\in\mathcal N(q)$, then 
\[
\lim_{n\ra\f}h(\vs_{q,n})=\lim_{n\ra\f}h(\us_{q,n}).
\]
\end{proposition}
  
\begin{proof}
Let $\al(q)=(a_i)$ and $m\in\mathcal N(q)$ such that $l=l(m)>0$.  Write $n=m(l+1)+r$ as in (\ref{eq:kong-1}), and $\su=a_1\ldots a_n=\sv(\overline{\sv}^+)^l\sw$. For $k>n$ we take an arbitrary word $\sy:=y_1\dots y_k$ in $B_k(\us_{q,n})$, and a subword $y_{i+1}\dots y_{i+N}$ of length $N:=[n/2]$. For convenience, and without loss of generality, we assume that $k$ is a multiple of $N$.
Let us consider the possible subwords $x_{i+1}\dots x_{i+N}$ of words $\sx=x_1\dots x_k$ with $f_{n,k}(\sx)=\sy$. Two such words are of course $y_{i+1}\dots y_{i+N}$ and $\overline{y_{i+1}\dots y_{i+N}}$. However, it is also possible that $x_{i+1}\dots x_{i+N}$ contains an occurrence of $\sv$ (or $\overline{\sv}$) that is the beginning of an occurrence of $\su$ (or $\overline{\su}$) and is therefore replaced by the map $F$ with $\sv^-$ (or $\overline{\sv}^+$). {Note that later iterations of $F$ do not change this block, in view of Lemma \ref{lem:24}.}  Since there are at most $N-m$ possible starting points for $\sv$ (or $\overline{\sv}$) and $m\geq 1$, it follows that there are at most $2N$ possible subwords $x_{i+1}\dots x_{i+N}$ which get mapped by $f_{n,k}$ to $y_{i+1}\dots y_{i+N}$.   

Applying this argument to each of the $k/N$ blocks $y_1\dots y_N, y_{N+1}\dots y_{2N},\dots,y_{k-N+1}\dots y_k$, we conclude that there are at most
$(2N)^{k/N}$ different words $\sx\in B_k(\vs_{q,n})$ with $f_{n,k}(\sx)=\sy$. Thus, the map $f_{n,k}$ is at most $(2N)^{k/N}$-to-one. It follows that
\[
\#B_k(\vs_{q,n})\leq (2N)^{k/N} \#B_k(\us_{q,n}),
\]
and so
\begin{equation*}
\frac{\log\#B_k(\vs_{q,n})}{k}\leq \frac{\log\#B_k(\us_{q,n})}{k}+\frac{\log 2N}{N}.
\end{equation*}
Letting $k\to\infty$ we get
\begin{equation}
h(\vs_{q,n})\leq h(\us_{q,n})+\frac{\log 2N}{N}\leq h(\us_{q,n})+\frac{\log n}{[n/2]}.
\label{eq:entropy-comparison}
\end{equation}
Hence, if there are infinitely many $m\in\mathcal{N}(q)$ with $l(m)>0$, then there are also infinitely many $n\in\mathcal{N}(q)$ such that \eqref{eq:entropy-comparison} holds. We can then let $n\to\infty$ along a suitable subsequence in $\mathcal N(q)$, and conclude that
\begin{equation*}
\lim_{n\to\infty} h(\vs_{q,n})\leq \lim_{n\to\infty} h(\us_{q,n}),
\label{eq:limit-equality}
\end{equation*}
using the fact that $h(\us_{q,n})$ is nondecreasing in $n$, and $h(\vs_{q,n})$ is nonincreasing in $n$.
\end{proof}

\subsection{Construction of $f_{n,k}$: the second case.}

Next, we assume that $q\in\bb^L$ and $l(m)=0$ for all but finitely many $m\in\mathcal N(q)$. 
Let 
\[
m_1:=\min\set{m\in\mathcal N(q): l(m')=0\textrm{ for all }m'\in\mathcal N(q)\textrm{ with }m'\ge m}.
\]
Note that $m_1\in\mathcal N(q)$.  Write $\sv_1:=a_1\ldots a_{m_1}$. Then $\sv_1$ is primitive, and $\sv_1(\overline{\sv_1}^+)^\f\prec \al(q)$. So, since $l(m_1)=0$, in view of (\ref{eq:kong-1}) there exists a word $\sw_1$ of shortest length $r_1:=|\sw_1|\ge 1$ such that 
$
\sv_1\sw_1 0^\f\succ\sv_1(\overline{\sv_1}^+)^\f.
$
Note that $1\le r_1=|\sw_1|\le |\sv_1|=m_1$. Define
$
\sv_2:=\sv_1\sw_1.
$
Then by Lemma \ref{lem:23} (ii) with $m=m_1, l=0$ and $r=r_1$ it follows that $m_2:=|\sv_2|\in\mathcal N(q)\cap[m_1, \f)$. By Lemma \ref{lem:22} this implies that $\sv_2$ is primitive and $\sv_{2}(\overline{\sv_2}^+)^\f\prec \al(q)$. 

Repeating the above argument we construct a sequence of words $(\sv_i)$ such that for each $i\ge 1$ the word $\sv_i$ is primitive and $\sv_i(\overline{\sv_i}^+)^\f\prec (a_i)$. Furthermore, for each $i\ge 1$,
\[
\sv_{i+1}=\sv_i\sw_i\quad \textrm{with}\quad  1\leq r_i:=|\sw_i|\le|\sv_i|=:m_i,
\]
and
\begin{equation}
\sw_i 0^\f \succ (\overline{\sv_i}^+)^\f.
\label{eq:w-is-big}
\end{equation}
Therefore, 
\begin{equation}\label{eq:alpha-sw}
\al(q)=(a_i)=\sv_1\sw_1\sw_2\sw_3\ldots=\sv_i\sw_i\sw_{i+1}\sw_{i+2}\ldots, \qquad i\geq 1.
\end{equation}
%for all $i\ge 1$. 

Clearly $|\sw_i|\ge 1$ for all $i\ge 1$. %So, the sequence  $(|\sw_i|)$ can not be strictly decreasing for all but finitely many $i$. In other words,  
Hence there are infinitely many integers $i$ such that $|\sw_i|\le |\sw_{i+1}|$. Observe also by Lemma \ref{lem:23} (ii) that $\sv_{i+1}=\sv_i\sw_i$ is primitive. This implies that $\sw_i^-\lge \overline{a_1\ldots a_{r_i}}$ for each $i\in\N$. It follows that one of the following cases must hold:
\begin{enumerate}[(i)]
\item $|\sw_i|<|\sw_{i+1}|$ for infinitely many $i$;  or
\item $\sw_i\succ \overline{a_1\dots a_{r_i}}^+$ for infinitely many $i$; or
\item there is $s\in\N$ such that $\al(q)=\sv_s \sw_s^\f=a_1\dots a_{m_s}(\overline{a_1\ldots a_{r_s}}^+)^\f$.  
\end{enumerate}
We consider the first two cases together; the third case, however, requires a different approach.

\medskip

{\noindent\bf Case A:} $|\sw_i|<|\sw_{i+1}|$ for infinitely many $i$, or $\sw_i\succ \overline{a_1\dots a_{r_i}}^+$ for infinitely many $i$.

\medskip

Fix an integer $s$ such that $|\sw_s|<|\sw_{s+1}|$ or $\sw_s\succ \overline{a_1\dots a_{r_s}}^+$. Set 
\[
n:=m_{s+2}=m_{s+1}+r_{s+1},
\] 
and write
 \[
 \su:=a_1\ldots a_n=\sv_s\sw_s\sw_{s+1}=\sv_{s+1}\sw_{s+1}, \qquad \sv:=\sv_{s+1}.
 \]
 Fix an integer $k>n$. With $\su$ and $\sv$ as above, set $F_A:=F_{\su,\sv}$ (see Definition \ref{def:general-F}). We first show that $F_A$ maps $B_k(\vs_{q,n})$ into itself.
 
\begin{lemma}\label{lem:F-A}
  For any $\sx\in B_k(\vs_{q,n})$ we have $F_A(\sx)\in B_k(\vs_{q,n})$, and  
\[
i_\su(F_A(\sx))\geq i_\su(\sx)+m_{s+1}\geq i_\su(\sx)+\frac{n}{2}.
\]
 \end{lemma}

 \begin{proof}
 The proof is similar to that of Lemma \ref{lem:24}.
Let $i:=i_\su(\sx)$. By symmetry we may assume that $x_{i+1}\ldots x_{i+n}=\su$. Then 
$
 \sx= x_1\ldots x_i\sv_{s+1}\sw_{s+1}x_{i+n+1}\ldots x_k.
$
 Write $F_A(\sx)=y_1\ldots y_k$. By Definition \ref{def:general-F} we have 
 \begin{align} 
\begin{split} 
y_1\ldots y_k &=x_1\ldots x_i\sv_{s+1}^-\overline{\sw_{s+1}x_{i+n+1}\ldots x_k}\\
&=x_1\ldots x_i\sv_{s}\sw_s^-\overline{\sw_{s+1}x_{i+n+1}\ldots x_k}.
\end{split} 
 \label{eq:yixi-1}
\end{align}
 Since the entire block $x_{i+m_{s+1}+1}\ldots x_k$ is being reflected by $F_A$, it suffices to show that for each $0\le j<i+m_{s+1}$ we have $\overline{\su}\prec y_{j+1}\ldots y_{j+n}\prec \su$.
  On one hand, by (\ref{eq:yixi-1})  we have $y_1\ldots y_{i+m_{s+1}}=x_1\ldots x_{i+m_{s+1}}^-$. Then 
  using the minimality of $i$ it follows that $y_{j+1}\ldots y_{j+n}\prec \su$ for all $j<i+m_{s+1}$. So, it remains to prove 
  \begin{equation} \label{eq:08-1}
  y_{j+1}\ldots y_{j+n}\succ \overline{\su}\quad\textrm{for all }0\le j<i+m_{s+1}.
  \end{equation}
  
  First, the minimality of $i$ implies that (\ref{eq:08-1}) holds for all $j<i+m_{s+1}-n=i-r_{s+1}$. The verification of (\ref{eq:08-1}) for $i-r_{s+1}\le j<i+m_{s+1}$ is split into the following {four} cases (see Figure \ref{Fig2}).
  
\begin{figure}[h!]
\begin{center}
\begin{tikzpicture}[ 
    scale=8,
    axis/.style={very thick, ->},
    important line/.style={thick},
    dashed line/.style={dashed, thin},
    pile/.style={thick, ->, >=stealth', shorten <=2pt, shorten
    >=2pt},
    every node/.style={color=black}
    ]
    % axis
 
    % Lines
 
    \draw[important line]  (0,0)--(1,0);
                \draw[important line] (0,-0.01)--(0, 0.01);
                 \draw[important line] (0.3,-0.01)--(0.3, 0.01);
                           \draw[important line] (0.5,-0.01)--(0.5, 0.01);  
                                     \draw[important line] (0.7,-0.01)--(0.7, 0.01);  
                                               \draw[important line] (1,-0.01)--(1, 0.01);  
      \node[] at (0,-0.05){$i-r_{s+1}$};
        \node[] at (0.3, -0.05){$i$};
         \node[] at (0.5, -0.05){$i+m_s$};
          \node[] at (0.7, -0.05){$i+m_{s+1}$};
           \node[] at (1, -0.05){$i+n$};
        
           \node[] at (0.15, 0.05){$x_{i-r_{s+1}}\ldots x_i$};
                  \node[] at (0.4, 0.05){$\sv_s$};
                   \node[] at (0.6, 0.05){$\sw_s^-$};
             \node[] at (0.85, 0.05){$\overline{\sw_{s+1}}$};
            
   \end{tikzpicture} 
\end{center}
\caption{The presentation of $y_{i-r_{s+1}}\ldots y_{i+n}=x_{i-r_{s+1}}\ldots x_i\sv_s\sw_s^-\overline{\sw_{s+1}}$.}
\label{Fig2}
\end{figure}
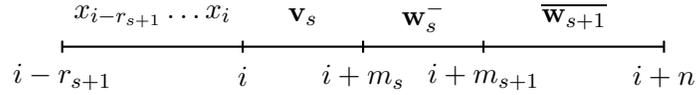

  \begin{itemize}
  \item[(I).] $j=i-r_{s+1}$. Then $y_{j+1}\ldots y_{j+n}=x_{j+1}\ldots x_i a_1\ldots a_{m_{s+1}}^-$. Note that $x_{j+1}\ldots x_i\lge \overline{a_1\ldots a_{i-j}}=\overline{a_1\ldots a_{r_{s+1}}}$. By Lemma \ref{lem:23} (i) it follows that 
  \[
  y_{i+1}\ldots y_{j+n}=a_1\ldots a_{m_{s+1}}^-\succ \overline{a_{r_{s+1}+1}\ldots a_n}. 
  \]
  This establishes (\ref{eq:08-1}) for $j=i-r_{s+1}$. 
  
  \item[(II).] $i-r_{s+1}<j<i$. Then $y_{j+1}\ldots y_i=x_{j+1}\ldots x_i\lge\overline{a_1\ldots a_{i-j}}$. Note that $i-j<r_{s+1}\le m_{s+1}$. Then by the primitivity of $\sv_{s+1}=a_1\ldots a_{m_{s+1}}$ it follows that 
  \[
  y_{i+1}\ldots y_{j+m_{s+1}}=a_1\ldots a_{j+m_{s+1}-i}\succ\overline{a_{i-j+1}\ldots a_{m_{s+1}}}.
  \]
  This proves (\ref{eq:08-1}) for $i-r_{s+1}<j<i$.
  
  \item[(III).] $i\le j<i+m_s$. Then (\ref{eq:08-1}) follows from the primitivity of $\sv_s=a_1\ldots a_{m_s}$, which implies 
  \[
  y_{j+1}\ldots y_{i+m_s}=a_{j-i+1}\ldots a_{m_s}\succ\overline{a_{1}\ldots a_{m_s-j+i}}.
  \]
  
  \item[(IV).] $i+m_s\le j<i+m_{s+1}$. Let  $t=j-(i+m_s)$. Then $0\le t<m_{s+1}-m_s=r_s$.  Note that $y_{j+1}\ldots y_{i+{m_{s+1}}}$ is a suffix of $\sw_s^-$.  Then by (\ref{eq:yixi-1}) and \eqref{eq:kong-1} it follows that 
  \begin{equation}\label{eq:30-1}
  y_{j+1}\ldots y_{i+m_{s+1}}\lge \overline{a_{t+1}\ldots a_{r_s}}\lge \overline{a_1\ldots a_{r_s-t}},
  \end{equation}
  where the second inequality follows since $a_1\ldots a_{r_s}$ is primitive by Lemma \ref{lem:23} (iii).
Note that, in view of \eqref{eq:kong-1}, the first inequality in \eqref{eq:30-1} is in fact strict if $\sw_s\succ \overline{a_1\dots a_{r_s}}^+$, so in this case we are done. Otherwise, we have $r_{s+1}=|\sw_{s+1}|>|\sw_s|=r_s$, so by (\ref{eq:yixi-1}) it follows that 
  \begin{equation}\label{eq:31-1}
  y_{i+m_{s+1}+1}\ldots y_{j+r_{s+1}}=a_1\ldots a_{r_{s+1}-r_s+t}\succ \overline{a_{r_s-t+1}\ldots a_{r_{s+1}}}.
  \end{equation}
  Here the inequality in (\ref{eq:31-1}) follows since $a_1\ldots a_{r_{s+1}}$ is primitive by Lemma \ref{lem:23} (iii).
Combining (\ref{eq:30-1}) and (\ref{eq:31-1}) we obtain (\ref{eq:08-1}) for $i+m_s\le j<i+m_{s+1}$.
\end{itemize} 
	We have now shown \eqref{eq:08-1} for all $0\leq j<i+m_{s+1}$. Hence, the proof is complete.
 \end{proof}

As a result of Lemma \ref{lem:F-A}, for some large enough $j$ (with $j<k$) we have $F_A^k(\sx)=\dots=F_A^{j+1}(\sx)=F_A^j(\sx)$. We now define
\[
	f_{n,k}^A(\sx):=F_A^k(\sx), \qquad \sx\in B_k(\vs_{q,n}).
\]
 
By Lemma \ref{lem:F-A}, $f_{n,k}^A$ maps $B_k(\vs_{q,n})$ into $B_k(\us_{q,n})$.  The next proposition now follows from a similar argument as in the proof of Proposition \ref{prop:left-continuity-positive}. % by using Lemma \ref{lem:23'}.
	
\begin{proposition}\label{prop:left-continuity-zero-1}
Let $q\in\bb^L$ with $\al(q)=\sv_1\sw_1\sw_2\sw_3\ldots$ satisfying \eqref{eq:w-is-big}. If $|\sw_i|<|\sw_{i+1}|$ for infinitely many $i$ or $\sw_i\succ \overline{a_1\dots a_{r_i}}^+$ for infinitely many $i$, then 
 $
 \lim_{n\ra\f}h(\vs_{q,n})=\lim_{n\ra\f}h(\us_{q,n}).
 $
\end{proposition}  
  
\medskip

{\noindent\bf Case B:} There is $s\in\N$ such that $\al(q)=\sv_s \sw_s^\f=a_1\dots a_{m_s}(\overline{a_1\ldots a_{r_s}}^+)^\f$.

\medskip

Note by the definition of $(\sv_i)$ that
$\sv_{s+j}=\sv_s\sw_s^j$ for any $j\in\N$.  Then by {(\ref{eq:alpha-sw})}, %and (\ref{eq:09-1}),
\begin{equation}\label{eq:alpha}
\al(q)=\sv_s\sw_s^\f=\sv_t\sw_s^\f=a_{1}\ldots a_{m_t}(\overline{a_1\ldots a_{r_s}}^+)^\f\quad\textrm{for any }t>s. 
\end{equation}
 
\begin{lemma} \label{lem:213}
Let $q\in\bb^L$ such that  $\al(q)=\sv_s\sw_s^\f=a_1\dots a_{m_s}(\overline{a_1\ldots a_{r_s}}^+)^\f$. Then 
\[
\al(q)\lge a_1\ldots a_{r_s}(\overline{a_1\ldots a_{r_s}}^+)^\f.
\]
\end{lemma}

\begin{proof}
Suppose on the contrary that $\al(q)=(a_i)\prec a_1\ldots a_{r_s}(\overline{a_1\ldots a_{r_s}}^+)^\f$. Then
\begin{equation}\label{eq:09-2}
(a_1\ldots a_{r_s}^-)^\f\prec \al(q)\prec a_1\ldots a_{r_s}(\overline{a_1\ldots a_{r_s}}^+)^\f,
\end{equation}
where the first inequality follows since $m_s\geq r_s$.
By Lemma \ref{lem:23} (iii), $a_1\ldots a_{r_s}$ is primitive. Thus, the same argument as in the proof of \cite[Proposition 3.9]{Allaart_Baker_Kong_2017} shows that $\al(q)\in X_{\mathcal G}$, where $X_{\mathcal G}$ is the subshift of finite type represented by the labeled graph $\mathcal G$ in Figure \ref{fig:1} (with $r:=r_s$).

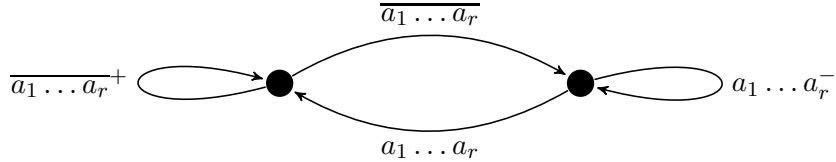
\begin{figure}[h!]
  \centering
 \begin{tikzpicture}[->,>=stealth',shorten >=1pt,auto,node distance=4cm,
                    semithick]

  \tikzstyle{every state}=[minimum size=0pt,fill=black,draw=none,text=black]

  \node[state] (A)                    { };
  \node[state]         (B) [ right of=A] { };

  \path[->,every loop/.style={min distance=0mm, looseness=60}]
   (A) edge [loop left,->]  node {$\overline{a_1\ldots a_{r}}^+$} (A)
            edge  [bend left]   node {$\overline{a_1\ldots a_{r}}$} (B)

        (B) edge [loop right] node {$a_1\ldots a_{r}^-$} (B)
            edge  [bend left]            node {$a_1\ldots a_{r}$} (A);
\end{tikzpicture}
  \caption{The picture of the labeled graph $\mathcal G$.}\label{fig:1}
\end{figure}
 
 Since $\al(q)=\sv_s\sw_s^\f$ ends with $\sw_s^\f=(\overline{a_1\ldots a_{r_s}}^+)^\f$, it follows from Figure \ref{fig:1} that 
 \[
 \si^j(\al(q))=a_1\ldots a_{r_s}(\overline{a_1\ldots a_{r_s}}^+)^\f
 \]
for some $j\geq 1$.
But then (\ref{eq:09-2}) gives $\si^j(\al(q)) \succ\al(q)$, contradicting Lemma \ref{l21}.
\end{proof}

In view of Lemma \ref{lem:213} we first consider the case $\al(q)=a_1\ldots a_{r_s}(\overline{a_1\ldots a_{r_s}}^+)^\f$. Here we could not find a suitable mapping $f_{n,k}$; instead we use a different method, based on ideas from \cite{Allaart_Baker_Kong_2017}.

 \begin{lemma} \label{lem:left-continuity-zero-3}
 Let $q\in\bb^L$ with $\al(q)=a_1\ldots a_{r}(\overline{a_1\ldots a_{r}}^+)^\f$, where $a_1\ldots a_{r}$ is primitive. Then 
\begin{equation}
\lim_{n\ra\f}h(\vs_{q,n})=\lim_{n\ra\f}h(\us_{q,n}). 
\label{equal-limits-bis}
\end{equation}
 \end{lemma}

\begin{proof}
We claim that
\begin{equation} \label{eq:09-3}
h(\vs_q)=\frac{\log 2}{r}.
\end{equation} 
 
First, observe by Figure \ref{fig:1} that $X_{\mathcal G}\subseteq \vs_q$, and hence
\begin{equation}
h(\vs_q)\ge h(X_{\mathcal G})=\frac{\log 2}{r}.
\label{eq:below-by-XG}
\end{equation}
On the other hand, by Lemma \ref{lem:22} (ii) it follows that $(a_1\ldots a_{r}^-)^\f\lle \al(q_{KL})$, so 
%\begin{equation}
$h(\us_{q,r})=0$.
%\label{eq:zero-entropy}
%\end{equation}
Furthermore, by the argument from the proof of \cite[Proposition 3.9]{Allaart_Baker_Kong_2017}, any sequence in $\vs_q$ is either itself in $\us_{q,r}$ or else consists of a finite (possibly empty) prefix from $\us_{q,r}$ followed by a sequence from $X_\mathcal{G}$. Hence, by a standard argument, 
\begin{equation*}
h(\vs_q)\leq h(X_\mathcal{G}).
\end{equation*}
Combined with \eqref{eq:below-by-XG}, this yields \eqref{eq:09-3}.

Next, for $n\in\N$ let $q_n<q$ be the base such that $\al(q_n)=(a_1\ldots a_{r}(\overline{a_1\ldots a_{r}}^+)^n\overline{a_1\ldots a_{r}})^\f$. Then $q_n\nearrow q$ as $n\ra\f$. Let $X_{\mathcal G, n}$ be the set of those sequences $(x_i)\in X_{\mathcal G}$ containing neither the word $a_1\ldots a_{r}(\overline{a_1\ldots a_{r}}^+)^n$ nor its reflection. Then $X_{\mathcal G, n}\subseteq \us_{q_n}$, and (see \cite[Lemma 4.2]{Allaart_Baker_Kong_2017})
 \[
 h(X_{\mathcal G, n})=\frac{\log \varphi_n}{r},
 \]
 where $\varphi_n$ is the unique positive root of $1+x+\cdots+x^{n-1}=x^n$. Since $\varphi_n\nearrow 2$ as $n\ra\f$, by (\ref{eq:09-3}) it follows that 
 \[
{h(\us_{q_n})}\ge h(X_{\mathcal G, n})=\frac{\log \varphi_n}{r}\to\;\frac{\log 2}{r}=h(\vs_q).
 \]
This establishes the left-continuity of $H$ at $q$. To obtain the stronger result \eqref{equal-limits-bis}, 
note that $\us_{q_n}=\us_{q, (n+2)r}$. Hence, $\lim_{n\ra\f}h(\us_{q,n})\ge h(\vs_q)\ge h(\us_q)$. The reverse inequality is obvious, since $\us_{q,n}\subseteq\vs_q$ for all $n\ge 1$. We conclude that
 \[
 \lim_{n\ra\f}h(\us_{q,n})=h(\vs_q)=h(\us_q)=\lim_{n\ra\f}h(\vs_{q,n}),
 \]
 where the last equality follows from the right continuity of $H$ (see Proposition \ref{prop:right-continuity}). 
 \end{proof}
 
 Finally, we consider the case that $(a_i)=\al(q)\succ a_1\ldots a_{r_s}(\overline{a_1\ldots a_{r_s}}^+)^\f$. Then there exists an integer $\ell\ge 1$ such that 
 \begin{equation}\label{eq:09-4}
  a_1\ldots a_{(\ell+1) r_s}\succ a_1\ldots a_{r_s}(\overline{a_1\ldots a_{r_s}}^+)^\ell.
 \end{equation}
  Note by (\ref{eq:alpha}) that $\al(q)=\sv_t\sw_s^\f=a_1\ldots a_{m_t}(\overline{a_1\ldots a_{r_s}}^+)^\f$ for any $t>s$. Take $t\in\N$ such that $m_t>{\ell} r_s$.  Write for $n:=m_{t+\ell+1}=m_t+(\ell+1)r_s$ that 
  \[
  \su:=a_1\ldots a_n=\sv_t\sw_s^{\ell+1}=a_1\ldots a_{m_t}(\overline{a_1\ldots a_{r_s}}^+)^{\ell+1}.
  \]
	Furthermore, put
	\[
	\sv:=\sv_t \sw_s=\sv_{t+1}.
	\]
	With $\su$ and $\sv$ as above, define $F_B:=F_{\su,\sv}$ (see Definition \ref{def:general-F}).

Using (\ref{eq:09-4}) and by a similar reasoning as in the proof of Lemma \ref{lem:F-A} it can be shown that $F_B$ maps $B_k(\vs_{q,n})$ into $B_k(\vs_{q,n})$. Furthermore, the earliest possible occurrence of $\su$ or $\overline{\su}$ in $F_B(\sx)$ starts later than the earliest occurrence of $\su$ or $\overline{\su}$ in $\sx$.
This implies that, for some large enough $j$ (with $j<k$) we have $F_B^k(\sx)=\dots=F_B^{j+1}(\sx)=F_B^j(\sx)$. 
We now define
\[
f_{n,k}^B(\sx):=F_B^k(\sx), \qquad \sx\in B_k(\vs_{q,n}).
\]

By the above argument, $f_{n,k}^B$ maps $B_k(\vs_{q,n})$ into $B_k(\us_{q,n})$. 
Note that the length $|\sv|=|\sv_t\sw_s|=m_t+r_s\ge n/2$ (since $m_t>\ell r_s$). As in the proof of Proposition \ref{prop:left-continuity-positive} we can now prove that the map $f_{n,k}^B$ is at most $(2N)^{k/N}$-to-one, where $N:=[n/2]$. This gives

 \begin{lemma} \label{lem:left-continuity-zero-4}
 Let $q\in\bb^L$ with $\al(q)=\sv_s\sw_s^\f=a_1\ldots a_{m_s}(\overline{a_1\ldots a_{r_s}}^+)^\f$ for some $s\ge 1$. If 
 $\al(q)\succ a_1\ldots a_{r_s}(\overline{a_1\ldots a_{r_s}}^+)^\f$, then 
 $
 \lim_{n\ra\f}h(\vs_{q,n})=\lim_{n\ra\f}h(\us_{q,n}).
 $
 \end{lemma}

 Combining Lemmas \ref{lem:213}, \ref{lem:left-continuity-zero-3} and \ref{lem:left-continuity-zero-4} we obtain:

 \begin{proposition} \label{prop:left-continuity-zero-2}
 Let $q\in\bb^L$ with $\al(q)=\sv_s\sw_s^\f=a_1\ldots a_{m_s}(\overline{a_1\ldots a_{r_s}}^+)^\f$ for some $s\ge 1$. Then 
 $
 \lim_{n\ra\f}h(\vs_{q,n})=\lim_{n\ra\f}h(\us_{q,n}).
 $
 \end{proposition}
 
 \begin{proof}[Proof of Theorem \ref{th:left-continuity}]
The theorem follows from  Propositions \ref{prop:left-continuity-positive}, \ref{prop:left-continuity-zero-1} and \ref{prop:left-continuity-zero-2}.  
\end{proof}

\section{Proof of Theorem \ref{thm:main2}} \label{sec:proof2}

We will use the following lemma for the Hausdorff dimension under H\"{o}lder continuous maps (cf.~\cite{Falconer_1990}).

\begin{lemma}\label{l26}
Let $f: (X, \rho_X)\ra (Y, \rho_Y)$ be a H\"{o}lder map between two metric spaces, i.e., there exist constants $C>0$ and $\xi>0$ such that
\[
\rho_Y(f(x), f(x'))\le C \rho_X(x, x')^\xi\quad\textrm{for any }x, x'\in X.
\]
Then $\dim_H f(X)\le \frac{1}{\xi}\dim_H X.$
\end{lemma}

It will be convenient to introduce a family of (mutually equivalent) metrics $\{\rho_q:q>1\}$ {on $\Omega$} defined by
\begin{equation*}
\rho_q((c_i),(d_i)):=q^{-\inf\{i\ge 1:c_i\neq d_i\}}, \qquad q>1.
\end{equation*}
Then $(\Omega, \rho_q)$ is a compact metric space.
Let $\dim_H^{(q)}$ denote Hausdorff dimension on $\Omega$ with respect to the metric $\rho_q$.
For $p>1$ and $q>1$,
\begin{equation*}
\rho_q((c_i),(d_i))=\rho_p((c_i),(d_i))^{\log q/\log p},
\end{equation*}
and {by Lemma \ref{l26}} this gives the useful relationship 
\begin{equation}
\dim_H^{(p)}E=\frac{\log q}{\log p}\dim_H^{(q)}E, \qquad E\subseteq\Omega.
\label{eq:24}
\end{equation}

The following result is well known (see \cite[Lemma 2.7]{Jordan-Shmerkin-Solomyak-2011} or \cite[Lemma 2.2]{Allaart-2017b}):

\begin{lemma} \label{l27}
For each $q\in(1, M+1)$, the map $\pi_q$ is Lipschitz on $(\Omega,\rho_q)$, and the restriction
\begin{equation*}
\pi_q: (\us_q,\rho_q)\to (\u_q,|.|);\qquad \pi_q((x_i))=\sum_{i=1}^\infty \frac{x_i}{q^i}
\end{equation*}
is bi-Lipschitz, where $|.|$ denotes the Euclidean metric on $\R$. In particular,
\[
\dim_H \u_q=\dim_H^{(q)} \us_q.
\]
\end{lemma}

Lastly, we need an analog of Proposition \ref{prop:same-entropy} for Hausdorff dimension.

\begin{lemma} \label{lem:equal-dimension}
For every $q\in(1,M+1]$, $\dim_H^{(q)} \us_q=\dim_H^{(q)} \vs_q$.
\end{lemma}

\begin{proof}
The proof is similar to that of Proposition \ref{prop:same-entropy}, but easier: By \eqref{eq:26} and the countable stability of Hausdorff dimension, we have $\dim_H^{(q)} \us_q\leq \dim_H^{(q)} \vs_q$. The reverse inequality follows since $\vs_q\backslash \us_q$ is countable.
\end{proof}

\begin{proof}[Proof of Theorem \ref{thm:main2}]
Consider first the case when $\alpha(q)=(a_1\dots a_m^-)^\infty$ for a primitive word $a_1\dots a_m$. Here $\vs_q$ can be written in finite terms as
\[
\vs_q=\{(x_i)\in\Omega: \overline{a_1\dots a_m}\preceq x_{n+1}\dots x_{n+m}\preceq a_1\dots a_m\ \mbox{for all}\ n\geq 0\},
\]
so $\vs_q$ is a subshift of finite type. It is well known (see, for instance, \cite{Kenyon_Peres_1991}) that the Hausdorff dimension of a subshift of finite type is given by its topological entropy. Thus, using Lemmas \ref{l26} and \ref{l27}, we obtain
\[
\dim_H \u_q=\dim_H^{(q)} \us_q=\dim_H^{(q)} \vs_q=\frac{h(\vs_q)}{\log q}=\frac{h(\us_q)}{\log q},
\]
where we also used Proposition \ref{prop:same-entropy} and Lemma \ref{lem:equal-dimension}.
 
Next, let $q\in\overline{\ub}$ and write $\al(q)=(a_i)$. Then by Lemma \ref{lem:admissible-words} there is a sequence of points $(q_n: n\in\N)$ such that $q_n$ increases to $q$ and for each $n$, $\alpha(q_n)=(a_1\dots a_{m_n}^-)^\f$ for some integer $m_n$ such that $a_1\dots a_{m_n}$ is primitive. So by the first case above and Theorem \ref{thm:main1},
\[
\dim_H^{(q)} \us_q \geq \dim_H^{(q)} \us_{q_n}=\frac{h(\us_{q_n})}{\log q}\to \frac{h(\us_q)}{\log q}.
\]
On the other hand, for any set $E\subseteq \Omega^\N$ we have $\dim_H^{(q)}\leq h(E)/\log q$, and so
\[
\dim_H^{(q)} \us_q \leq \frac{h(\us_q)}{\log q}.
\]
Hence, by Lemma \ref{l27},
\[
\dim_H \u_q=\dim_H^{(q)} \us_q=\frac{h(\us_q)}{\log q}.
\]

Finally, let $q\in(q_{KL},M+1]\backslash \overline{\ub}$. Then $q$ lies in a connected component $(q_0,q_1)$ of $(q_{KL},M+1]\backslash \overline{\ub}$. It was shown in \cite{Kong_Li_2015} that $h(\us_q)$ is constant on $[q_0,q_1)$ and 
\[
\dim_H \u_q=\frac{h(\vs_{q_0})}{\log q}=\frac{h(\us_{q_0})}{\log q}=\frac{h(\us_q)}{\log q}.
\]
This completes the proof.
\end{proof}

%\bibliographystyle{abbrv}
%\bibliography{Continuity}

\end{document}